\documentclass[11pt,a4paper]{amsart}

\usepackage{amsmath}
\usepackage{mathtools}
\usepackage{amsthm}
\usepackage{amssymb}
\usepackage{hyperref}

\usepackage{booktabs}
\usepackage{wrapfig}

\newcommand*\fullref[3][\relax]{%
  \ifdefined\hyperref%
    {\hyperref[#3]{#2\penalty 200\ \ref*{#3}#1}}%
  \else%
    {#2\penalty 200\ \relax\ref{#3}#1}%
  \fi%
}

\usepackage{tikz}
\usetikzlibrary{arrows.meta}
\usetikzlibrary{calc}
\usetikzlibrary{decorations.pathmorphing}
\usetikzlibrary{decorations.text}

\usetikzlibrary{shapes.geometric} 
\usetikzlibrary{shapes.misc}  

\tikzset{
  bst/.style={
    standard/.style={
      font=\small,
      draw=gray,
      rounded rectangle,
      minimum width=4.5mm,
      minimum height=4.5mm,
      inner xsep=0mm,
      inner ysep=1mm,
      outer sep=0mm,
      line width=.5pt,
    },
    empty/.style={
      minimum width=3mm,
      minimum height=3mm,
    },
    triangle/.style={
      isosceles triangle,
      isosceles triangle apex angle=60,
      shape border rotate=90,
      rounded corners=2mm,
      minimum width=8mm,
      inner xsep=0mm,
      inner ysep=.5mm
    },
    blank/.style={
      draw=none,
    },
    nodecount/.style={
      blank,
      font=\scriptsize,
    },
    every node/.style={standard},
    every child/.style={draw=black,line width=.6pt},
    level distance=10mm,
    level 1/.style={sibling distance=60mm},
    level 2/.style={sibling distance=30mm},
    level 3/.style={sibling distance=15mm},
  },
  medbst/.style={
    bst,
    level distance=10mm,
    level 1/.style={sibling distance=15mm},
    level 2/.style={sibling distance=15mm},
    level 3/.style={sibling distance=15mm},
  },
  smallbst/.style={
    bst,
    level distance=8mm,
    level 1/.style={sibling distance=10mm},
    level 2/.style={sibling distance=10mm},
    level 3/.style={sibling distance=10mm},
  },
  tinybst/.style={
    bst,
    level distance=5mm,
    level 1/.style={sibling distance=8mm},
    level 2/.style={sibling distance=8mm},
    level 3/.style={sibling distance=8mm},
    every node/.append style={
      font=\footnotesize,
    },
    triangle/.append style={
      rounded corners=1mm,
      minimum width=7mm,
      inner xsep=-.5mm,
    },
  },
  microbst/.style={
    bst,
    standard/.append style={
      font=\scriptsize,
      minimum width=3mm,
      minimum height=3mm,
      inner ysep=.25mm,
    },
    level distance=3mm,
    level 1/.style={sibling distance=6mm},
    level 2/.style={sibling distance=6mm},
    level 3/.style={sibling distance=6mm},
  },
  nanobst/.style={
    bst,
    standard/.append style={
      font=\tiny,
      minimum width=2mm,
      minimum height=2mm,
      inner ysep=.25mm,
    },
    level distance=2mm,
    level 1/.style={sibling distance=4mm},
    level 2/.style={sibling distance=4mm},
    level 3/.style={sibling distance=4mm},
  },
}
\usetikzlibrary{intersections}

\makeatletter

\pgfkeys{/pgf/.cd,
  rectangle corner radius/.initial=3pt
}
\newif\ifpgf@rectanglewrc@donecorner@
\def\pgf@rectanglewithroundedcorners@docorner#1#2#3#4{%
  \edef\pgf@marshal{%
    \noexpand\pgfintersectionofpaths
      {%
        \noexpand\pgfpathmoveto{\noexpand\pgfpoint{\the\pgf@xa}{\the\pgf@ya}}%
        \noexpand\pgfpathlineto{\noexpand\pgfpoint{\the\pgf@x}{\the\pgf@y}}%
      }%
      {%
        \noexpand\pgfpathmoveto{\noexpand\pgfpointadd
          {\noexpand\pgfpoint{\the\pgf@xc}{\the\pgf@yc}}%
          {\noexpand\pgfpoint{#1}{#2}}}%
        \noexpand\pgfpatharc{#3}{#4}{\cornerradius}%
      }%
    }%
  \pgf@process{\pgf@marshal\pgfpointintersectionsolution{1}}%
  \pgf@process{\pgftransforminvert\pgfpointtransformed{}}%
  \pgf@rectanglewrc@donecorner@true
}
\pgfdeclareshape{rectangle with rounded corners}
{
  \inheritsavedanchors[from=rectangle] 
  \inheritanchor[from=rectangle]{north}
  \inheritanchor[from=rectangle]{north west}
  \inheritanchor[from=rectangle]{north east}
  \inheritanchor[from=rectangle]{center}
  \inheritanchor[from=rectangle]{west}
  \inheritanchor[from=rectangle]{east}
  \inheritanchor[from=rectangle]{mid}
  \inheritanchor[from=rectangle]{mid west}
  \inheritanchor[from=rectangle]{mid east}
  \inheritanchor[from=rectangle]{base}
  \inheritanchor[from=rectangle]{base west}
  \inheritanchor[from=rectangle]{base east}
  \inheritanchor[from=rectangle]{south}
  \inheritanchor[from=rectangle]{south west}
  \inheritanchor[from=rectangle]{south east}

  \savedmacro\cornerradius{%
    \edef\cornerradius{\pgfkeysvalueof{/pgf/rectangle corner radius}}%
  }

  \backgroundpath{%
    \northeast\advance\pgf@y-\cornerradius\relax
    \pgfpathmoveto{}%
    \pgfpatharc{0}{90}{\cornerradius}%
    \northeast\pgf@ya=\pgf@y\southwest\advance\pgf@x\cornerradius\relax\pgf@y=\pgf@ya
    \pgfpathlineto{}%
    \pgfpatharc{90}{180}{\cornerradius}%
    \southwest\advance\pgf@y\cornerradius\relax
    \pgfpathlineto{}%
    \pgfpatharc{180}{270}{\cornerradius}%
    \northeast\pgf@xa=\pgf@x\advance\pgf@xa-\cornerradius\southwest\pgf@x=\pgf@xa
    \pgfpathlineto{}%
    \pgfpatharc{270}{360}{\cornerradius}%
    \northeast\advance\pgf@y-\cornerradius\relax
    \pgfpathlineto{}%
  }

  \anchor{before north east}{\northeast\advance\pgf@y-\cornerradius}
  \anchor{after north east}{\northeast\advance\pgf@x-\cornerradius}
  \anchor{before north west}{\southwest\pgf@xa=\pgf@x\advance\pgf@xa\cornerradius
    \northeast\pgf@x=\pgf@xa}
  \anchor{after north west}{\northeast\pgf@ya=\pgf@y\advance\pgf@ya-\cornerradius
    \southwest\pgf@y=\pgf@ya}
  \anchor{before south west}{\southwest\advance\pgf@y\cornerradius}
  \anchor{after south west}{\southwest\advance\pgf@x\cornerradius}
  \anchor{before south east}{\northeast\pgf@xa=\pgf@x\advance\pgf@xa-\cornerradius
    \southwest\pgf@x=\pgf@xa}
  \anchor{after south east}{\southwest\pgf@ya=\pgf@y\advance\pgf@ya\cornerradius
    \northeast\pgf@y=\pgf@ya}

  \anchorborder{%
    \pgf@xb=\pgf@x
    \pgf@yb=\pgf@y%
    \southwest%
    \pgf@xa=\pgf@x
    \pgf@ya=\pgf@y%
    \northeast%
    \advance\pgf@x by-\pgf@xa%
    \advance\pgf@y by-\pgf@ya%
    \pgf@xc=.5\pgf@x
    \pgf@yc=.5\pgf@y%
    \advance\pgf@xa by\pgf@xc
    \advance\pgf@ya by\pgf@yc%
    \edef\pgf@marshal{%
      \noexpand\pgfpointborderrectangle
      {\noexpand\pgfqpoint{\the\pgf@xb}{\the\pgf@yb}}
      {\noexpand\pgfqpoint{\the\pgf@xc}{\the\pgf@yc}}%
    }%
    \pgf@process{\pgf@marshal}%
    \advance\pgf@x by\pgf@xa%
    \advance\pgf@y by\pgf@ya%
    \pgfextract@process\borderpoint{}%
    \pgf@rectanglewrc@donecorner@false
    \southwest\pgf@xc=\pgf@x\pgf@yc=\pgf@y
    \advance\pgf@xc\cornerradius\relax\advance\pgf@yc\cornerradius\relax
    \borderpoint
    \ifdim\pgf@x<\pgf@xc\relax\ifdim\pgf@y<\pgf@yc\relax
      \pgf@rectanglewithroundedcorners@docorner{-\cornerradius}{0pt}{180}{270}%
    \fi\fi
    \ifpgf@rectanglewrc@donecorner@\else
      \southwest\pgf@yc=\pgf@y\relax\northeast\pgf@xc=\pgf@x\relax
      \advance\pgf@xc-\cornerradius\relax\advance\pgf@yc\cornerradius\relax
      \borderpoint
      \ifdim\pgf@x>\pgf@xc\relax\ifdim\pgf@y<\pgf@yc\relax
       \pgf@rectanglewithroundedcorners@docorner{0pt}{-\cornerradius}{270}{360}%
      \fi\fi
    \fi
    \ifpgf@rectanglewrc@donecorner@\else
      \northeast\pgf@xc=\pgf@x\relax\pgf@yc=\pgf@y\relax
      \advance\pgf@xc-\cornerradius\relax\advance\pgf@yc-\cornerradius\relax
      \borderpoint
      \ifdim\pgf@x>\pgf@xc\relax\ifdim\pgf@y>\pgf@yc\relax
       \pgf@rectanglewithroundedcorners@docorner{\cornerradius}{0pt}{0}{90}%
      \fi\fi
    \fi
    \ifpgf@rectanglewrc@donecorner@\else
      \northeast\pgf@yc=\pgf@y\relax\southwest\pgf@xc=\pgf@x\relax
      \advance\pgf@xc\cornerradius\relax\advance\pgf@yc-\cornerradius\relax
      \borderpoint
      \ifdim\pgf@x<\pgf@xc\relax\ifdim\pgf@y>\pgf@yc\relax
       \pgf@rectanglewithroundedcorners@docorner{0pt}{\cornerradius}{90}{180}%
      \fi\fi
    \fi
  }
}

\makeatother

\makeatletter

\@ifpackagelater{tikz}{2013/12/01}{
  \newcommand{\convexpath}[2]{
    [
    create hullcoords/.code={
      \global\edef\namelist{#1}
      \foreach [count=\counter] \nodename in \namelist {
        \global\edef\numberofnodes{\counter}
        \coordinate (hullcoord\counter) at (\nodename);
      }
      \coordinate (hullcoord0) at (hullcoord\numberofnodes);
      \pgfmathtruncatemacro\lastnumber{\numberofnodes+1}
      \coordinate (hullcoord\lastnumber) at (hullcoord1);
    },
    create hullcoords
    ]
    ($(hullcoord1)!#2!-90:(hullcoord0)$)
    \foreach [
    evaluate=\currentnode as \previousnode using \currentnode-1,
    evaluate=\currentnode as \nextnode using \currentnode+1
    ] \currentnode in {1,...,\numberofnodes} {
      let \p1 = ($(hullcoord\currentnode) - (hullcoord\previousnode)$),
      \n1 = {atan2(\y1,\x1) + 90},
      \p2 = ($(hullcoord\nextnode) - (hullcoord\currentnode)$),
      \n2 = {atan2(\y2,\x2) + 90},
      \n{delta} = {Mod(\n2-\n1,360) - 360}
      in
      {arc [start angle=\n1, delta angle=\n{delta}, radius=#2]}
      -- ($(hullcoord\nextnode)!#2!-90:(hullcoord\currentnode)$)
    }
  }
}{
  \newcommand{\convexpath}[2]{
    [
    create hullcoords/.code={
      \global\edef\namelist{#1}
      \foreach [count=\counter] \nodename in \namelist {
        \global\edef\numberofnodes{\counter}
        \coordinate (hullcoord\counter) at (\nodename);
      }
      \coordinate (hullcoord0) at (hullcoord\numberofnodes);
      \pgfmathtruncatemacro\lastnumber{\numberofnodes+1}
      \coordinate (hullcoord\lastnumber) at (hullcoord1);
    },
    create hullcoords
    ]
    ($(hullcoord1)!#2!-90:(hullcoord0)$)
    \foreach [
    evaluate=\currentnode as \previousnode using \currentnode-1,
    evaluate=\currentnode as \nextnode using \currentnode+1
    ] \currentnode in {1,...,\numberofnodes} {
      let \p1 = ($(hullcoord\currentnode) - (hullcoord\previousnode)$),
      \n1 = {atan2(\x1,\y1) + 90},
      \p2 = ($(hullcoord\nextnode) - (hullcoord\currentnode)$),
      \n2 = {atan2(\x2,\y2) + 90},
      \n{delta} = {Mod(\n2-\n1,360) - 360}
      in
      {arc [start angle=\n1, delta angle=\n{delta}, radius=#2]}
      -- ($(hullcoord\nextnode)!#2!-90:(hullcoord\currentnode)$)
    }
  }
}

\makeatother%

\tikzset{
  elementoutline/.style={
    draw=gray,
    fill=white,
    thick,
    rectangle with rounded corners,
    inner sep=.4mm,
  },
}

\tikzset{
  mogrifyarrow/.style={
    ->,
    >/.tip=Computer Modern Rightarrow,
    decorate,
    decoration={
      zigzag,
      amplitude=0.2em,
      segment length=0.35em,
      pre length=0.35em,
      post length=0.35em,
    },
  },
}

\tikzset{
  bstoutline/.style={darkgray,thick}
}

\tikzset{
  olsubword/.style={
    every node/.append style={
      name=mainnode,
      draw=darkgray,
      rounded corners=.5mm,
      inner sep=.5mm,
    },
    baseline=(mainnode.base),
  }
}

\newcommand*\olsubword[1]{\tikz[olsubword]\node{#1};}

\theoremstyle{definition}
\newtheorem{definition}{Definition}[section]

\newtheorem{algorithm}[definition]{Algorithm}

\newtheorem{example}[definition]{Example}

\theoremstyle{plain}

\newtheorem{lemma}[definition]{Lemma}
\newtheorem{proposition}[definition]{Proposition}

\numberwithin{equation}{section}
\newcommand*{\textparens}[1]{\textup{(}#1\textup{)}}

\newcommand*{\defterm}[1]{\emph{#1}}

\makeatletter
\newcommand\chyph{\penalty\@M-\hskip\z@skip}
\makeatother

\DeclarePairedDelimiter{\parens}{\lparen}{\rparen}

\DeclarePairedDelimiter{\set}{\{}{\}}
\DeclarePairedDelimiterX{\gset}[2]{\{}{\}}{\,#1:#2\,}

\makeatletter
\newcommand{\sizeddelimiter}[2]{\bBigg@{#1}#2}
\makeatother

\DeclarePairedDelimiterX{\pres}[2]{\langle}{\rangle}{#1\,\delimsize\vert\,\mathopen{}#2}

\newcommand*{\drel}[1]{\mathcal{#1}}

\newcommand*{\aA}{\mathcal{A}}

\newcommand*{\cochseq}{\mathrm{cochseq}}

\newcommand*{\plac}{{\mathsf{plac}}}
\newcommand*{\hypo}{{\mathsf{hypo}}}
\newcommand*{\sylv}{{\mathsf{sylv}}}

\newcommand*{\baxt}{{\mathsf{baxt}}}
\newcommand*{\stal}{{\mathsf{stal}}}
\newcommand*{\taig}{{\mathsf{taig}}}

\newcommand*{\sylvcong}{\equiv_\sylv}

\newcommand*{\plit}{\mathrm{P}}

\newcommand*{\psylv}[2][]{\plit_{\sylv}\parens[#1]{#2}}

\newcommand{\cyc}{\sim}

\begin{document}

\title[Combinatorics of cyclic shifts]{Combinatorics of cyclic shifts in plactic, hypoplactic, sylvester, and related monoids}

\author{Alan J. Cain}
\address{%
Centro de Matem\'{a}tica e Aplica\c{c}\~{o}es\\
Faculdade de Ci\^{e}ncias e Tecnologia\\
Universidade Nova de Lisboa\\
2829--516 Caparica\\
Portugal
}
\email{%
a.cain@fct.unl.pt
}
\thanks{The first author was supported by an Investigador {\sc FCT} fellowship ({\sc IF}/01622/2013/{\sc CP}1161/{\sc
    CT}0001).}

\author{Ant\'onio Malheiro}
\address{%
Departamento de Matem\'{a}tica \& Centro de Matem\'{a}tica e Aplica\c{c}\~{o}es\\
Faculdade de Ci\^{e}ncias e Tecnologia\\
Universidade Nova de Lisboa\\
2829--516 Caparica\\
Portugal
}
\email{%
ajm@fct.unl.pt
}
\thanks{For both authors, this work was partially supported by by the Funda\c{c}\~{a}o para a Ci\^{e}ncia e a
    Tecnologia (Portuguese Foundation for Science and Technology) through the project {\scshape UID}/{\scshape
      MAT}/00297/2013 (Centro de Matem\'{a}tica e Aplica\c{c}\~{o}es), and the project {\scshape PTDC}/{\scshape
      MHC-FIL}/2583/2014.}

\begin{abstract}
  The cyclic shift graph of a monoid is the graph whose vertices are elements of the monoid and whose edges link
  elements that differ by a cyclic shift. For certain monoids connected with combinatorics, such as the
  plactic monoid (the monoid of Young tableaux) and the sylvester monoid (the monoid of binary search trees), connected
  components consist of elements that have the same evaluation (that is, contain the same number of each generating
  symbol). This paper discusses new results on the diameters of connected components of the cyclic shift graphs of the
  finite-rank analogues of these monoids, showing that the maximum diameter of a connected component is dependent only on the
  rank. The proof techniques are explained in the case of the sylvester monoid.
\end{abstract}

\maketitle

\section{Introduction}

In a monoid $M$, two elements $s$ and $t$ are related by a cyclic shift, denoted $s \sim t$, if and only if there exist
$x,y \in M$ such that $s = xy$ and $t = yx$. In the plactic monoid (the monoid of Young tableaux, denoted $\plac$; see
\cite[Ch.~5]{lothaire_algebraic}), elements that have the same evaluation (that is, which contain the same number of
each symbol) can be obtained from each other by iterated application of cyclic shifts
\cite[\S~4]{lascoux_plaxique}. Furthermore, in the plactic monoid of rank $n$ (denoted $\plac_n$), it is known that
$2n-2$ applications of cyclic shifts are sufficient \cite[Theorem~17]{choffrut_lexicographic}.

To restate these results in a new form, define the \defterm{cyclic shift graph} $K(M)$ of a monoid $M$ to be the
undirected graph with vertex set $M$ and, for all $s,t \in M$, an edge between $s$ and $t$ if and only if $s \sim t$.
Connected components of $K(M)$ are $\sim^*$-classes (where $\sim^*$ is the reflexive and transitive closure of $\sim$),
since they consist of elements that are related by iterated cyclic shifts. Thus the results discussed above say that
each connected component of $K(\plac)$ consists of precisely the elements with a given evaluation, and that the diameter
of a connected component of $K(\plac_n)$ is at most $2n-2$. Note that connected components are of unbounded size,
despite there being a bound on diameters that is dependent only on the rank.

The plactic monoid is celebrated for its ubiquity, appearing in many diverse contexts (see the discussion and references in
\cite[Ch.~5]{lothaire_algebraic}). It is, however, just one member of a family of `plactic-like' monoids that are closely
connected with combinatorics. These monoids include the hypoplactic monoid (the monoid of quasi-ribbon tableaux)
\cite{krob_noncommutative4,novelli_hypoplactic}, the sylvester monoid (binary search trees) \cite{hivert_sylvester}, the
taiga monoid (binary search trees with multiplicities) \cite{priez_lattice}, the stalactic monoid (stalactic tableaux)
\cite{hivert_commutative,priez_lattice}, and the Baxter monoid (pairs of twin binary search trees)
\cite{giraudo_baxter2}. These monoids, including the plactic monoid, arise in a parallel way. For each monoid, there is
a so-called insertion algorithm that allows one to compute a combinatorial object (of the corresponding type) from a
word over the infinite ordered alphabet $\aA = \set{1 < 2 < 3 <\ldots}$; the relation that relates pairs of words that
give the same combinatorial object is a congruence (that is, it is compatible with multiplication in the free monoid
$\aA^*$). The monoid arises by factoring the free monoid $\aA^*$ by this congruence; thus elements of the monoid
(equivalence classes of words) are in one-to-one correspondence with the combinatorial objects. \fullref{Table}{tbl:monoids}
lists these monoids and their corresponding objects.

\begin{table}[t]
  \centering
  \caption{Monoids and corresponding combinatorial objects.}
  \label{tbl:monoids}
  \begin{tabular}{llll}
    \toprule
    \textit{Monoid} & \textit{Symbol} & \textit{Combinatorial object}          & \textit{Citation}               \\
    \midrule
    Plactic         & $\plac$       & Young tableau                          & \cite[ch.~5]{lothaire_algebraic} \\
    Hypoplactic     & $\hypo$       & Quasi-ribbon tableau                   & \cite{novelli_hypoplactic}      \\
    Stalactic       & $\stal$       & Stalactic tableau                      & \cite{hivert_commutative}       \\
    Sylvester       & $\sylv$       & Binary search tree                     & \cite{hivert_sylvester}         \\
    Taiga           & $\taig$       & Binary search tree with multiplicities & \cite[\S~5]{priez_lattice}            \\
    Baxter          & $\baxt$       & Pair of twin binary search trees       & \cite{giraudo_baxter2}          \\
    \bottomrule
  \end{tabular}
\end{table}

\begin{table}[t]
  \centering
  \caption{Properties of connected component of the cyclic shift graph for rank-$n$ monoids: whether they are characterized by evaluation, and known values and bounds for their maximum diameters.}
  \label{tbl:components}
  \begin{tabular}{lccccc}
    \toprule
                    &                       & \multicolumn{4}{c}{\textit{Maximum diameter}}                                                     \\
    \cmidrule(lr){3-6}
                    &                       &                      &                        & \multicolumn{2}{c}{\textit{Known bounds}}         \\
    \cmidrule(lr){5-6}
    \textit{Monoid} & \textit{Char. by evaluation} & \textit{Known value} & \textit{Conjecture}    & \textit{Lower}                   & \textit{Upper} \\
    \midrule
    $\plac_n$       & Y                     & ?                    & $n-1$                  & $n-1$                            & $2n-3$         \\
    $\hypo_n$       & Y                     & $n-1$                & ---                    & ---                              & ---            \\
    $\stal_n$       & N                     & $\begin{cases} n-1   & \text{if $n < 3$} \\
                                                             n     & \text{if $n \geq 3$}\end{cases}$ & --- & --- & ---                         \\
    $\sylv_n$       & Y                     & ?                    & $n-1$                  & $n-1$                            & $n$            \\
    $\taig_n$       & Y                     & ?                    & $n-1$                  & $n-1$                            & $n$            \\
    $\baxt_n$       & N                     & ?                    & ?                      & ?                                & ?              \\
    \bottomrule
  \end{tabular}
\end{table}

Analogous questions arise for the cyclic shift graph of each of these monoids. In a forthcoming paper
\cite{cm_cyclicshifts2}, the present authors make a comprehensive study of connected components in the cyclic shift
graphs of each of these monoids. For several of these monoids, it turns out that each connected component of its cyclic
shift graph consists of precisely the elements with a given evaluation, and that the diameters of connected component in
the rank-$n$ case are bounded by a quantity dependent on the rank. (Again, it should be emphasized that there is no
bound on the \emph{size} of these connected components.) In each case, the authors either establish the exact value of
the maximum diameter or give bounds; \fullref{Table}{tbl:components} summarizes the results from
\cite{cm_cyclicshifts2}. Also, although these monoids are multihomogeneous (words in $\aA^*$ representing the same
element contain the same number of each symbol), the authors also exhibit a rank~$4$ multihomogeneous monoid for which
there is no bound on the diameter of connected components. Thus it seems to be the underlying
combinatorial objects that ensure the bound on diameters.  This also is of interest because cyclic shifts are a possible
generalization of conjugacy from groups to monoids; thus the combinatorial objects are here linked closely to the
algebraic structure of the monoid.

The present paper illustrates these results by focussing on the sylvester monoid (denoted $\sylv$ or $\sylv_n$ in the
rank-$n$ case). (The authors previously proved that each connected component of $K(\sylv)$ consists of precisely the
elements with a given evaluation \cite[\S~3]{cm_conjugacy}.) \fullref{Section}{sec:sylv} recalls the definition and
necessary facts about the sylvester monoid. \fullref{Section}{sec:lowerbound} gives a complete proof that there is a connected
component in $K(\sylv_n)$ with diameter at least $n-1$; this establishes the lower bound on the maximum diameter shown
in \fullref{Table}{tbl:components}. The complete proof that every connected component of $K(\sylv_n)$ has diameter at
most $n$, establishing the upper maximum diameter shown in \fullref{Table}{tbl:components}, is very long and complicated. Thus
\fullref{Section}{sec:upperbound} gives the proof for connected components consisting of elements that contain each
symbol from $\set{1,\ldots,k}$ (for some $k$) exactly once; this avoids many of the complexities of the general case.

\section{Binary search trees and the sylvester monoid}
\label{sec:sylv}

This section gathers the relevant definitions and background on the sylvester monoid; see \cite{hivert_sylvester} for
further reading.

Recall that $\aA$ denotes the infinite ordered alphabet $\set{1 < 2 < \ldots}$. Fix a natural number $n$ and let
$\aA_n = \set{1 < 2 < \ldots < n}$ be the set of the first $n$ natural numbers, viewed as a finite ordered alphabet. A
word $u \in \aA^*$ is \defterm{standard} if it contains each symbol in $\set{1,\ldots,|u|}$ exactly once.

A \defterm{\textparens{right strict} binary search tree} (BST) is a rooted binary tree labelled by symbols from $\aA$,
where the label of each node is greater than or equal to the label of every node in its left subtree, and strictly less
than the label of every node in its right subtree. An example of a binary search tree is:
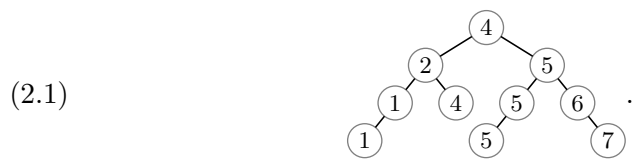
\begin{equation}
\label{eq:bsteg}
\begin{tikzpicture}[tinybst,baseline=-10mm]
  \node (root) {$4$}
    child[sibling distance=16mm] { node (0) {$2$}
      child { node (00) {$1$}
        child { node (000) {$1$} }
        child[missing]
      }
      child { node (01) {$4$} }
    }
    child[sibling distance=16mm] { node (1) {$5$}
      child { node (10) {$5$}
        child { node (100) {$5$} }
        child[missing]
      }
      child { node (11) {$6$}
        child[missing]
        child { node (111) {$7$} }
      }
    };
\end{tikzpicture}.
\end{equation}

The following algorithm inserts a new symbol into a BST, adding it as a leaf node in the unique place that maintains
the property of being a BST.

\begin{algorithm}
\label{alg:sylvinsertone}
\textit{Input:} A binary search tree $T$ and a symbol $a \in \aA_n$.

If $T$ is empty, create a node and label it $a$. If $T$ is non-empty, examine the label $x$ of the root
node; if $a \leq x$, recursively insert $a$ into the left subtree of the root node; otherwise recursively insert $a$
into the right subtree of the root note. Output the resulting tree.
\end{algorithm}

For $u \in \aA^*$, define $\psylv{u}$ to be the right strict binary search tree obtained by starting with the empty tree and inserting the
symbols of the word $u$ one-by-one using \fullref{Algorithm}{alg:sylvinsertone}, proceeding \emph{right-to-left} through $u$. For
example, $\psylv{5451761524}$ is \eqref{eq:bsteg}. Define the relation $\sylvcong$ by
\[
u \sylvcong v \iff \psylv{u} = \psylv{v},
\]
for all $u,v \in \aA^*$.  The relation $\sylvcong$ is a congruence, and the \defterm{sylvester monoid}, denoted $\sylv$,
is the factor monoid $\aA^*\!/{\sylvcong}$; the \defterm{sylvester monoid of rank $n$}, denoted $\sylv_n$, is the factor
monoid $\aA_n^*/{\sylvcong}$ (with the natural restriction of $\sylvcong$). Each element $[u]_{\sylvcong}$ (where
$u \in \aA^*$) can be identified with the binary search tree $\psylv{u}$. The monoid $\sylv$ is presented by
$\pres{\aA}{\drel{R}_\sylv}$, where
\[
\drel{R}_\sylv = \gset[\big]{(cavb,acvb)}{a \leq b < c,\; v \in \aA^*};
\]
the monoid $\sylv_n$ is presented by $\pres{\aA_n}{\drel{R}_\sylv}$, where the set of defining relations
$\drel{R}_\sylv$ is naturally restricted to $\aA_n^*\times \aA_n^*$. Notice that $\sylv$ and $\sylv_n$ are
multihomogeneous.

A \defterm{reading} of a binary search tree $T$ is a word formed from the symbols that appear in the nodes of $T$,
arranged so that the child nodes appear before parents. A word $w \in \aA^*$ is a reading of $T$ if and only if
$\psylv{w} = T$. The words in $[u]_{\sylvcong}$ are precisely the readings of $\psylv{u}$.

A binary search tree $T$ with $k$ nodes is \defterm{standard} if it has exactly one node labelled by each symbol in
$\set{1,\ldots,k}$; clearly $T$ is standard if and only if all of its readings are standard words.

The \defterm{left-to-right postfix traversal}, or simply the \defterm{postfix traversal}, of a rooted binary tree $T$ is
the sequence that `visits' every node in the tree as follows: it recursively perform the postfix traversal of the left
subtree of the root of $T$, then recursively perform the postfix traversal of the right subtree of the root of $T$, then
visits the root of $T$. The \defterm{left-to-right infix traversal}, or simply the \defterm{infix traversal}, of a
rooted binary tree $T$ is the sequence that `visits' every node in the tree as follows: it recursively performs the
infix traversal of the left subtree of the root of $T$, then visits the root of $T$, then recursively performs the infix
traversal of the right subtree of the root of $T$. Thus the postfix and infix traversals of any binary tree with the
same shape as \eqref{eq:bsteg} visit nodes as shown on the left and right below:
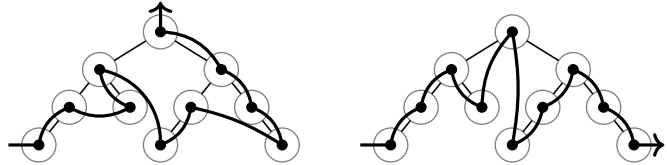
\[
\begin{tikzpicture}[tinybst,baseline=-7.5mm]
  \node (root) {}
    child[sibling distance=16mm] { node (0) {}
      child { node (00) {}
        child { node (000) {} }
        child[missing]
      }
      child { node (01) {} }
    }
    child[sibling distance=16mm] { node (1) {}
      child { node (10) {}
        child { node (100) {} }
        child[missing]
      }
      child { node (11) {}
        child[missing]
        child { node (111) {} }
      }
    };
  \begin{scope}[very thick,line cap=round]
    \draw (000.center) edge[bend left=30] (00.center);
    \draw (00.center) edge[bend right=30] (01.center);
    \draw (01.center) edge[bend left=30] (0.center);
    \draw (0.center) edge[bend left=40] (100.center);
    \draw (100.center) edge[bend right=30] (10.center);
    \draw (10.center) edge[bend left=10] (111.center);
    \draw (111.center) edge[bend right=30] (11.center);
    \draw (11.center) edge[bend right=30] (1.center);
    \draw (1.center) edge[bend right=30] (root.center);
  \end{scope}
  \draw[very thick] ($ (000.center) + (-4mm,0) $) -- (000.center);
  \draw[very thick,->] (root.center) -- ($ (root.center) + (0,4mm) $);
  \foreach\x in {000,00,01,0,100,10,111,11,1,root} {
    \draw[draw=black,fill=black] (\x.center) circle (.66mm);
  }
\end{tikzpicture}
\qquad
\begin{tikzpicture}[tinybst,baseline=-7.5mm]
  \node (root) {}
    child[sibling distance=16mm] { node (0) {}
      child { node (00) {}
        child { node (000) {} }
        child[missing]
      }
      child { node (01) {} }
    }
    child[sibling distance=16mm] { node (1) {}
      child { node (10) {}
        child { node (100) {} }
        child[missing]
      }
      child { node (11) {}
        child[missing]
        child { node (111) {} }
      }
    };
  \begin{scope}[very thick,line cap=round]
    \draw (000.center) edge[bend left=30] (00.center);
    \draw (00.center) edge[bend left=30] (0.center);
    \draw (0.center) edge[bend right=30] (01.center);
    \draw (01.center) edge[bend left=20] (root.center);
    \draw (root.center) edge[bend left=10] (100.center);
    \draw (100.center) edge[bend right=30] (10.center);
    \draw (10.center) edge[bend right=30] (1.center);
    \draw (1.center) edge[bend left=30] (11.center);
    \draw (11.center) edge[bend left=30] (111.center);
  \end{scope}
  \draw[very thick] ($ (000.center) + (-4mm,0) $) -- (000.center);
  \draw[very thick,->] (111.center) -- ($ (111.center) + (4mm,0) $);
  \foreach\x in {000,00,01,0,100,10,111,11,1,root} {
    \draw[draw=black,fill=black] (\x.center) circle (.66mm);
  }
\end{tikzpicture}
\]

The following result is immediate from the definition of a binary search tree, but it is used frequently:

\begin{proposition}
  \label{prop:infixreading}
  For any binary search tree $T$, if a node $x$ is encountered before a node $y$ in an infix traversal, then $x \leq y$.
\end{proposition}

In this paper, a \defterm{subtree} of a binary search tree will always be a rooted subtree.  Let $T$ be a binary search
tree and $x$ a node of $T$. The \defterm{complete subtree of $T$ at $x$} is the subtree consisting of $x$ and every node
below $x$ in $T$. The \defterm{path of left child nodes in $T$ from $x$} is the path that starts at $x$ and enters left
child nodes until a node with empty left subtree is encountered.

Let $B$ be a subtree of $T$. Then $B$ is said to be \defterm{on} the path of left child nodes from $x$ if the root of
$B$ is one of the nodes on this path. The \defterm{left-minimal} subtree of $B$ in $T$ is the complete subtree at the
left child of the left-most node in $B$; the \defterm{right-maximal} subtree of $B$ in $T$ is the complete subtree at
the right child of the right-most node in $B$.

In diagrams of binary search trees, individual nodes are shown as round, while subtrees as shown as triangles. An edge
emerging from the top of a triangle is the edge running from the root of that subtree to its parent node. A vertical
edge joining a node to its parent indicates that the node may be either a left or right child. An edge emerging from the
bottom-left of a triangle is the edge to that subtree's left-minimal subtree; an edge emerging from the bottom-right of
a triangle is the edge to that subtree's right-maximal subtree.

\section{Lower bound on diameters}
\label{sec:lowerbound}

Let $u \in \aA_n^*$ be a standard word. The \defterm{cocharge sequence} of $u$, denoted $\cochseq(u)$, is a sequence (of
length $n$) calculated from $u$ as follows:
\begin{enumerate}
\item Draw a circle, place a point $*$ somewhere on its circumference, and, starting from $*$, write $u$ anticlockwise
  around the circle.
\item Label the symbol $1$ with $0$.
\item Iteratively, after labelling some $i$ with $k$, proceed clockwise from $i$ to $i+1$. If the symbol $i+1$ is
  reached \emph{before} $*$, label $i+1$ by $k+1$. Otherwise, if the symbol $i+1$ is reached \emph{after} $*$, label
  $i+1$ by $k$.
\item The sequence whose $i$-th term is the label of $i$ is $\cochseq(u)$.
\end{enumerate}


\begin{wrapfigure}[9]{r}{30mm}
\begin{tikzpicture}
  \draw (0,0) circle[radius=5mm];
  %
  \draw[gray,thick,->] (-10:12.5mm) arc[radius=12.5mm,start angle=-10,end angle=-170];
  \draw[gray,decorate,decoration={text along path,text={Labelling},text color=gray,text align={align=center}}] (-170:16mm) arc[radius=16mm,start angle=-170,end angle=-10];
  \foreach\i/\ilabel in {0/*,9/1,8/2,7/4,6/6,5/3,4/7,3/5} {
    \node at ($ (90-\i*30:7mm) $) {$\ilabel$};
  };
  \foreach\i/\ilabel in {9/0,8/0,5/0,7/1,3/1,6/2,4/2} {
    \node[font=\footnotesize,gray] at ($ (90-\i*30:10mm) $) {$\ilabel$};
  };
\end{tikzpicture}
\end{wrapfigure}
For example, for the word $u = 1246375$, the labelling process is shown on the right, and it follows that
$\cochseq(u) = (0,0,0,1,1,2,2)$. Notice that the first term of a cocharge sequence is always $0$, and that each term in
the sequence is either the same as its predecessor or greater by $1$. Thus the $i$-th term in the sequence always lies
in the set $\set{0,1,\ldots,i-1}$.

The usual notion of `cocharge' is obtained by summing the cocharge sequence (see \cite[\S~5.6]{lothaire_algebraic}).

\begin{lemma}
  \label{lem:cochseqcycles}
  \begin{enumerate}
  \item Let $u \in \aA_n^*$ and $a \in \aA_n \setminus \set{1}$ be such that $ua$ is a standard word. Then $\cochseq(ua)$ is
    obtained from $\cochseq(au)$ by adding $1$ to the $a$-th component.
  \item Let $x,y \in \aA_n^*$ be such that $xy \in \aA_n^*$ is a standard word. Then corresponding components of $\cochseq(xy)$ and $\cochseq(yx)$
    differ by at most $1$.
  \end{enumerate}
\end{lemma}

\begin{proof}
  Consider how $a$ is labelled during the calculation of $\cochseq(ua)$ and $\cochseq(au)$:
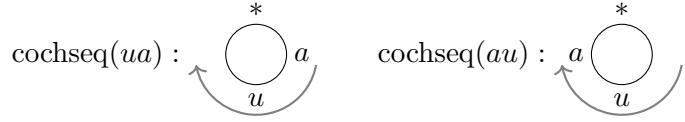
  \[
  \cochseq(ua):
  \begin{tikzpicture}[baseline=-1mm]
    \draw (0,0) circle[radius=4mm];
    \draw[gray,thick,->] (-10:8mm) arc[radius=8mm,start angle=-10,end angle=-170];
    %
    \foreach\i/\ilabel in {0/*,6/u,3/a} {
      \node at ($ (90-\i*30:6mm) $) {$\ilabel$};
    };
  \end{tikzpicture}
  \qquad
  \cochseq(au):
  \begin{tikzpicture}[baseline=-1mm]
    \draw (0,0) circle[radius=4mm];
    \draw[gray,thick,->] (-10:8mm) arc[radius=8mm,start angle=-10,end angle=-170];
    %
    \foreach\i/\ilabel in {0/*,9/a,6/u} {
      \node at ($ (90-\i*30:6mm) $) {$\ilabel$};
    };
  \end{tikzpicture}
  \]
  In the calculation of $\cochseq(ua)$, the symbol $a-1$ receives a label $k$, and then $a$ is reached \emph{after} $*$
  is passed; hence $a$ also receives the label $k$. In the calculation of $\cochseq(au)$, the symbols $1,\ldots,a-1$
  receive the same labels as they do in the calculation of $\cochseq(ua)$, but after labelling $a-1$ by $k$ the symbol
  $a$ is reached \emph{before} $*$ is passed; hence $a$ receives the label $k+1$; after this point, labelling proceeds
  in the same way. This proves part~1). For part~2), notice that one of $x$ and $y$ does not contain the symbol $1$; the
  result is now an immediate consequence of part~1).
\end{proof}

\begin{proposition}
  \label{prop:cochseqsylvester}
  Let $u,v \in \aA_n^*$ be standard words such that $u \sylvcong v$. Then $\cochseq(u) = \cochseq(v)$.
\end{proposition}

\begin{proof}
  It suffices to prove the result when $w$ and $w'$ differ by a single application of a defining relation
  $(cavb,acvb) \in \drel{R}_\sylv$ where $a \leq b < c$. In this case, $w = pcavbq$ and $w' = pacvbq$, where $p,q,v \in \aA_n^*$ and
  $a,b,c \in \aA_n$ with $a \leq b < c$. Since $w$ and $w'$ are standard words, $a < b$.

  Consider how labels are assigned to the symbols $a$, $b$, and $c$ when calculating the cocharge sequence of $w$:
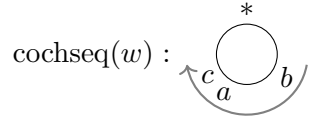
  \[
    \cochseq(w):
    \begin{tikzpicture}[baseline=-1mm]
      \draw (0,0) circle[radius=4mm];
      \draw[gray,thick,->] (-10:8mm) arc[radius=8mm,start angle=-10,end angle=-170];
      %
      \foreach\i/\ilabel in {0/*,8/c,7/a,4/b} {
        \node at ($ (90-\i*30:6mm) $) {$\ilabel$};
      };
    \end{tikzpicture}
  \]
  Among these three symbols, $a$ will receive a label first, then $b$, then $c$. Thus, after $a$, the labelling process
  will pass $*$ at least once to visit $b$ and only then visit $c$. Thus if we interchange $a$ and $c$, we do not alter
  the resulting labelling. Hence $\cochseq(w) = \cochseq(w')$.
\end{proof}

For any standard binary tree $T$ in $\sylv_n$, define $\cochseq(T)$ to be $\cochseq(u)$ for any standard word
$u \in \aA_n^*$ such that $T = \psylv{u}$. By \fullref{Proposition}{prop:cochseqsylvester}, $\cochseq(T)$ is well-defined.

\begin{proposition}
  \label{prop:sylvlowerbound}
  The connected component of $K(\sylv_n)$ consisting of standard elements has diameter at least $n-1$.
\end{proposition}

\begin{proof}
  Let $t = 12\cdots (n-1)n$ and $u = n(n-1)\cdots 21$, and let
  \[
    T = \psylv{t}=
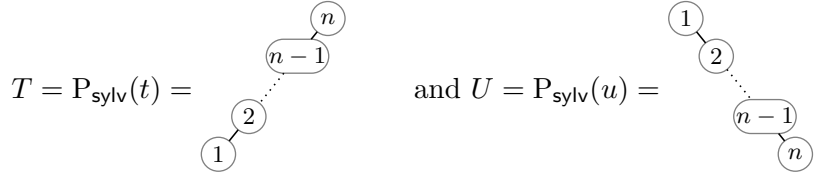
    \begin{tikzpicture}[tinybst,baseline=-10.5mm]
      \node {$n$}
      child { node {$n-1$}
        child[level distance=8mm,sibling distance=12.8mm,dotted] { node[solid] {$2$}
          child[level distance=5mm,sibling distance=8mm,solid] { node {$1$} }
          child[missing]
        }
        child[missing]
      }
      child[missing];
    \end{tikzpicture}
    \qquad\text{ and }U = \psylv{u} =
    \begin{tikzpicture}[tinybst,baseline=-10.5mm]
      \node {$1$}
      child[missing]
      child { node {$2$}
        child[missing]
        child[level distance=8mm,sibling distance=12.8mm,dotted] { node[solid] {$n-1$}
          child[missing]
          child[level distance=5mm,sibling distance=8mm,solid] { node {$n$} }
        }
      };
    \end{tikzpicture}
  \]
  Since $T$ and $U$ have the same evaluation, they are $\cyc^*$-related by \cite[\S~3]{cm_conjugacy}, and so in the same
  connected component of $K(\sylv_n)$.  Let $T = T_0,T_1,\ldots,T_{m-1},T_m = U$ be a path in $K(\sylv_n)$ from $T$ to
  $U$. Then for $i = 0,\ldots,m-1$, we have $T_i \cyc T_{i+1}$. That is, there are words $u_i,v_i \in \aA_n^*$ such that
  $T_i = \psylv{u_iv_i}$ and $T_{i+1} = \psylv{v_iu_i}$. By \fullref[(2)]{Lemma}{lem:cochseqcycles}, $\cochseq(T_i)$ and
  $\cochseq(T_{i+1})$ differ by adding $1$ or subtracting $1$ from certain components. Hence corresponding components of
  $\cochseq(T)$ and $\cochseq(U)$ differ by at most $m$. Since $\cochseq(T) = (0,0,\ldots,0,0)$ and
  $\cochseq(U) = (0,1,\ldots,n-2,n-1)$, it follows that $m \geq n-1$. Hence $T$ and $U$ are a distance at least $n-1$
  apart in $K(\sylv_n)$.
\end{proof}

\section{Upper bounds on diameters}
\label{sec:upperbound}

\begin{proposition}
  \label{proposition:sylvupperbound}
  Any two standard elements of $\sylv_n$ are a distance at most $n$ apart in $K(\sylv_n)$.
\end{proposition}

\begin{proof}
  Since $\sylv_m$ embeds into $\sylv_n$ for all $m \leq n$, and since $K(\sylv_m)$ is the subgraph of $K(\sylv_n)$
  induced by $\sylv_m$, this result follows from \fullref{Lemma}{lem:sylvupperbound} below.
\end{proof}

\fullref{Lemma}{lem:sylvupperbound} proves that in $K(\sylv_n)$ there is a path of length at most $n$ between two
standard elements with the same number of nodes. First, however, the strategy used to construct such a path is
illustrated in the following example.

\begin{example}
  \label{eg:sylvupperbound}
  Let
  \[
    T =
    \begin{tikzpicture}[baseline=(0.base)]
      \begin{scope}[tinybst]
        \node (root) {$4$}
        child { node (0) {$2$}
          child { node (00) {$1$} }
          child { node (01) {$3$} }
        }
        child { node (00) {$5$} };
      \end{scope}
    \end{tikzpicture},\quad
    U =
    \begin{tikzpicture}[baseline=(1.base)]
      \begin{scope}[tinybst]
        \node (root) {$1$}
        child[missing]
        child { node (1) {$4$}
          child { node (10) {$3$}
            child { node (100) {$2$} }
            child[missing]
          }
          child { node (11) {$5$} }
        };
      \end{scope}
    \end{tikzpicture}\;\; \in \sylv_5
  \]
  The aim is to build a sequence $T = T_0 \cyc T_1 \cyc T_2 \cyc T_3 \cyc T_4 \cyc T_5 = U$. Consider the postfix
  traversal of $U$. The $5$ steps in this traversal are shown below on the right, together with the relevant cases in
  the proof of \fullref{Lemma}{lem:sylvupperbound}. The parts of $U$ that have been visited already at each step are
  outlined. The idea is that the $h$-th cyclic shift leads to a tree $T_h$ where copies of the outlined parts of $U$
  appear on the path of left child nodes from the root of $T_h$. Note that cyclic shifts never break up the subwords
  (outlined) that represent the already-built subtrees. (The difficulty in the general proof is showing that a suitable
  cyclic shift exists at each step.)
\begin{align*}
T = T_0 ={}& \psylv{13254} \cyc{} \psylv{54132}\\
=T_1 ={}&
  \begin{tikzpicture}[baseline=(0.base)]
    \begin{scope}[tinybst]
      \node (root) {$2$}
      child { node (1) {$1$} }
      child { node (1) {$3$}
        child[missing]
        child { node (11) {$4$}
          child[missing]
          child { node (110) {$5$} }
        }
      };
      \draw[bstoutline] (root) circle[radius=3mm];
    \end{scope}
  \end{tikzpicture}
  &
    \begin{tikzpicture}[baseline=(1.base)]
      \begin{scope}[tinybst]
        \node (root) {$1$}
        child[missing]
        child { node (1) {$4$}
          child { node (10) {$3$}
            child { node[fill=gray] (100) {$2$} }
            child[missing]
          }
          child { node (11) {$5$} }
        };
      \end{scope}
      \draw[bstoutline] (100) circle[radius=3mm];
    \end{tikzpicture}
        &
          \quad\text{Base of induction}
  \\
  ={}& \psylv{5431\olsubword{2}}  \cyc{} \psylv{1\olsubword{2}543}
  \displaybreak[0]\\
=T_2 ={}&
  \begin{tikzpicture}[baseline=(0.base)]
    \begin{scope}[tinybst]
      \node (root) {$3$}
      child { node (0) {$2$}
        child { node (00) {$1$} }
        child[missing]
      }
      child { node (1) {$4$}
        child[missing]
        child { node (11) {$5$} }
      };
    \end{scope}
    \draw[bstoutline] \convexpath{root,0}{3mm};
  \end{tikzpicture}
  &
    \begin{tikzpicture}[baseline=(1.base)]
      \begin{scope}[tinybst]
        \node (root) {$1$}
        child[missing]
        child { node (1) {$4$}
          child { node[fill=gray] (10) {$3$}
            child { node (100) {$2$} }
            child[missing]
          }
          child { node (11) {$5$} }
        };
      \end{scope}
      \draw[bstoutline] \convexpath{10,100}{3mm};
    \end{tikzpicture}
        &
          \quad\text{Induction step, case 3}
  \\
  ={}& \psylv{541\olsubword{23}}   \cyc{} \psylv{41\olsubword{23}5}
  \displaybreak[0]\\
=T_3 ={}&
  \begin{tikzpicture}[baseline=(0.base)]
    \begin{scope}[tinybst]
      \node (root) {$5$}
      child { node (0) {$3$}
        child { node (00) {$2$}
          child { node (000) {$1$} }
          child[missing]
        }
        child { node (01) {$4$} }
      }
      child[missing];
    \end{scope}
    \draw[bstoutline] (root) circle[radius=3mm];
    \draw[bstoutline] \convexpath{0,00}{3mm};
  \end{tikzpicture}
  &
    \begin{tikzpicture}[baseline=(1.base)]
      \begin{scope}[tinybst]
        \node (root) {$1$}
        child[missing]
        child { node (1) {$4$}
          child { node (10) {$3$}
            child { node (100) {$2$} }
            child[missing]
          }
          child { node[fill=gray] (11) {$5$} }
        };
      \end{scope}
    \draw[bstoutline] (11) circle[radius=3mm];
    \draw[bstoutline] \convexpath{10,100}{3mm};
    \end{tikzpicture}
        &
          \quad\text{Induction step, case 1}
  \\
  ={}& \psylv{41\olsubword{23}\olsubword{5}}  \cyc{} \psylv{1\olsubword{23}\olsubword{5}4}
  \displaybreak[0]\\
=T_4 ={}&
  \begin{tikzpicture}[baseline=(0.base)]
    \begin{scope}[tinybst]
      \node (root) {$4$}
      child { node (0) {$3$}
        child { node (00) {$2$}
          child { node (000) {$1$} }
          child[missing]
        }
        child[missing]
      }
      child { node (1) {$5$} };
    \end{scope}
    \draw[bstoutline] \convexpath{root,1,00}{3mm};
  \end{tikzpicture}
  &
    \begin{tikzpicture}[baseline=(1.base)]
      \begin{scope}[tinybst]
        \node (root) {$1$}
        child[missing]
        child { node[fill=gray] (1) {$4$}
          child { node (10) {$3$}
            child { node (100) {$2$} }
            child[missing]
          }
          child { node (11) {$5$} }
        };
      \end{scope}
      \draw[bstoutline] \convexpath{1,11,100}{3mm};
    \end{tikzpicture}
        &
          \quad\text{Induction step, case 2}
  \\
  ={}& \psylv{1\olsubword{2354}}   \cyc{} \psylv{\olsubword{2354}1}
  \displaybreak[0]\\
=T_5 ={}&
    \begin{tikzpicture}[baseline=(1.base)]
      \begin{scope}[tinybst]
        \node (root) {$1$}
        child[missing]
        child { node (1) {$4$}
          child { node (10) {$3$}
            child { node (100) {$2$} }
            child[missing]
          }
          child { node (11) {$5$} }
        };
      \end{scope}
      \draw[bstoutline] \convexpath{root,11,100}{3mm};
    \end{tikzpicture} = U
  &
    \begin{tikzpicture}[baseline=(1.base)]
      \begin{scope}[tinybst]
        \node[fill=gray] (root) {$1$}
        child[missing]
        child { node (1) {$4$}
          child { node (10) {$3$}
            child { node (100) {$2$} }
            child[missing]
          }
          child { node (11) {$5$} }
        };
      \end{scope}
      \draw[bstoutline] \convexpath{root,11,100}{3mm};
    \end{tikzpicture}
        &
          \quad\text{Induction step, case 4}
\end{align*}
\end{example}

\begin{lemma}
  \label{lem:sylvupperbound}
  Let $T, U \in \sylv_n$ be standard and have $n$ nodes. Then there is a sequence $T = T_0, T_1, \ldots, T_n = U$ with
  $T_h \cyc T_{h+1}$ for $h = 0,\ldots,n-1$.
\end{lemma}

\begin{proof}
  This proof is only concerned with standard BSTs; thus for brevity nodes are identified with their labels. Notice that
  each of $T$ and $U$ has exactly one node labelled by each symbol in $\aA_n$.

  Consider the left-to-right postfix traversal of $U$; there are exactly $n$ steps in this traversal. Let $u_h$ be the
  node visited at the $h$-th step of this traversal.

  For $h = 1,\ldots,n$, let $U_h = \set{u_1,\ldots,u_h}$ and let $U_h^\top$ be the set of nodes in $U_h$ that do not
  lie below any other node in $U_h$. Since a later step in a postfix traversal is never below an earlier one, it follows
  that $u_h \in U_h^\top$ for all $h$. Let $B_h$ be the subtree of $U$ consisting of $u_h$ and every node that is
  below $u_h$.

  The aim is to construct inductively the required sequence. Let $h = 1,\ldots,n$ and suppose
  $U_h^\top = \set{u_{i_1},\ldots,u_{i_k}}$ (where $i_1 < \ldots < i_k = h$). Then the tree $T_h$ will satisfy the
  following conditions:
  \begin{itemize}
  \item[P1] The subtree $B_{i_k}$ appears at the root of $T_h$.
  \item[P2] The subtrees $B_{i_k},\ldots,B_{i_1}$ appear, in that order (but not necessarily consecutively), on the path
    of left child nodes from the root of $T_h$.
  \end{itemize}
  (Note that conditions P1 and P2 do not apply to $T_0$.)

  \medskip
  \noindent\textit{Base of induction.} Set $T_0 = T$. Take any reading of $T_0$ and factor it as $wu_{1}w'$. Let $T_1 = \psylv{w'wu_{1}}$. Clearly $T_1$ has
  root node $u_{1}$. Since $B_1$ consists only of the node $u_1$ (since $u_1$ is the first node in $U$ visited by the
  postfix traversal and is thus a leaf node), $T_1$ satisfies P1. Further, $T_1$ trivially satisfies P2. Finally, note
  that $T_0 \cyc T_1$. (For an illustration, see the definition of $T_1$ in \fullref{Example}{eg:sylvupperbound}.)

  \medskip
  \noindent\textit{Induction step.}  The remainder of the sequence of trees is built inductively. Suppose that the
  tree $T_h$ satisfies P1 and P2; the aim is to find $T_{h+1}$ satisfying P1 and P2 with $T_h \cyc T_{h+1}$. There are
  four cases, depending on the relative positions of $u_h$ and $u_{h+1}$ in $U$:
  \begin{enumerate}
  \item $u_h$ is a left child node and $u_{h+1}$ is in the right subtree of the parent of $u_h$;
  \item $u_h$ is the right child of $u_{h+1}$, and $u_{h+1}$ has non-empty left subtree;
  \item $u_h$ is the left child of $u_{h+1}$ (which implies, by the definition of the postfix traversal, that $u_{h+1}$
    has empty right subtree);
  \item $u_h$ is the right child of $u_{h+1}$, and $u_{h+1}$ has empty left subtree.
  \end{enumerate}

  \smallskip
  \noindent\textit{Case 1.} Suppose that, in $U$, the node $u_h$ is a left child node and $u_{h+1}$ is in the right subtree of the parent of
  $u_h$. (For an illustration of this case, see the step from $T_2$ to $T_3$ in \fullref{Example}{eg:sylvupperbound}.) Then $B_{h+1}$
  consists only of the node $u_{h+1}$, since by the definition of a postfix traversal $u_{h+1}$ is a leaf
  node. Furthermore, $U_{h+1}^\top = U_h^\top \cup \set{u_{h+1}}$.

  \begin{figure}[t]
    \centering
    \begin{tikzpicture}
      \begin{scope}[tinybst]
        \node[triangle] (aroot) at (0mm,-35mm) {$B_{\mathrlap{h}\,}$}
        child { node[triangle] (a0) {$\lambda$} }
        child { node[triangle] (a1) {$\delta$}
          child { node (a10) {$u_{i+1}$}
            child { node[triangle] (a100) {$\alpha$} }
            child { node[triangle] (a101) {$\beta$} }
          }
        };
      \end{scope}
      \begin{scope}[tinybst]
        \node (broot) at ($ (aroot) + (60mm,0mm) $) {$u_{i+1}$}
        child { node[triangle] (b0) {}
          child { node[triangle] (b00) {$B_{\mathrlap{h}\,}$}
            child { node[triangle] (b000) {$\lambda$} }
            child { node[triangle] (b001) {} }
          }
          child[missing]
        }
        child { node[triangle] (b1) {} };
      \end{scope}
      \node[draw,bstoutline,rounded rectangle,font=\footnotesize,minimum height=5.5mm,minimum width=10mm] (brootoutline) at (broot) {};
      \node (bhplus1) at ($ (broot) + (-12mm,0) $) {$B_{h+1}$};
      \draw[bstoutline] (bhplus1) -- (brootoutline);
      \node[anchor=east] at ($ (aroot) + (-15mm,0) $) {$T_h=$};
      \node[anchor=west] at ($ (broot) + (15mm,0) $) {$= T_{h+1}$};
      \draw ($ (aroot) + (20mm,0) $) edge[mogrifyarrow] node[fill=white,anchor=mid,inner xsep=.25mm,inner ysep=1mm] {$\cyc$} ($ (broot) + (-20mm,0) $);
    \end{tikzpicture}
    \caption{Induction step, case~1.}
    \label{fig:case1}
  \end{figure}
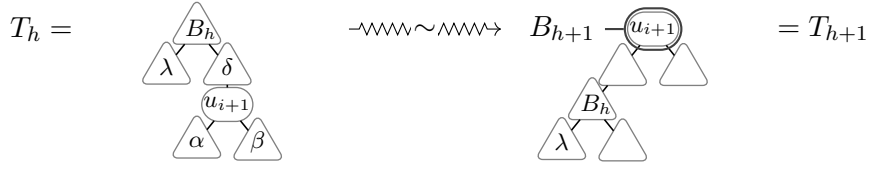%
  By P1, $B_h$ appears at the root of $T_h$. By \fullref{Proposition}{prop:infixreading} applied to $U$, the symbol $u_{h+1}$ is greater than
  every node of $B_h$, so $u_{h+1}$ must be in the right-maximal subtree of $B_h$ in $T_h$.

  As shown in \fullref{Figure}{fig:case1}, let $\lambda$ be a reading of the left-minimal subtree of $B_h$. Let $\delta$
  be a reading of the right-maximal subtree of $B_h$ outside of the complete subtree at $u_{h+1}$. Let $\alpha$ and
  $\beta$ be readings of the left and right subtrees of $u_{h+1}$, respectively. Note that the subtrees $B_i$ for
  $u_i \in U_h^\top$ are contained in $\lambda$.

  Thus $T_h = \psylv{\alpha\beta u_{h+1}\delta\lambda B_h}$. Let
  $T_{h+1} = \psylv{\delta\lambda B_h\alpha\beta u_{h+1}}$; note that $T_h \cyc T_{h+1}$.

  In computing $T_{h+1}$, the symbol $u_{h+1}$ is inserted first and becomes the root node. Since $B_{h+1}$ consists
  only of the node $u_{h+1}$, the tree $T_{h+1}$ satisfies P1. Since every symbol in $B_h$ and $\lambda$ is strictly
  less that every symbol in $\alpha$, $\beta$, or $\delta$, the trees $B_h$ and $\lambda$ are re-inserted on the path of
  left child nodes from the root of $T_{h+1}$. Thus all the subtrees $B_i$ for $u_i \in U_{h+1}^\top$ are on the
  path of left child nodes from the root, and so $T_{h+1}$ satisfies P2.

  \smallskip
  \noindent\textit{Case 2.} Suppose that in $U$, the node $u_h$ is the right child of $u_{h+1}$, and $u_{h+1}$ has non-empty left subtree. (For an illustration of this case, see the step from $T_3$ to $T_4$ in \fullref{Example}{eg:sylvupperbound}.)
  Let $u_g$ be the left child of $u_{h+1}$. Thus $B_{h+1}$ consists of $u_{h+1}$ with left subtree $B_g$ and right
  subtree $B_h$.  By the definition of the postfix traversal, $u_g$ was visited before the $h$-th step, but no node
  above $u_g$ has been visited. That is, $u_g \in U_h^\top$. Hence
  $U_{h+1}^\top = \parens[\big]{U_h^\top \setminus \set{u_g,u_h}} \cup \set{u_{h+1}}$.

  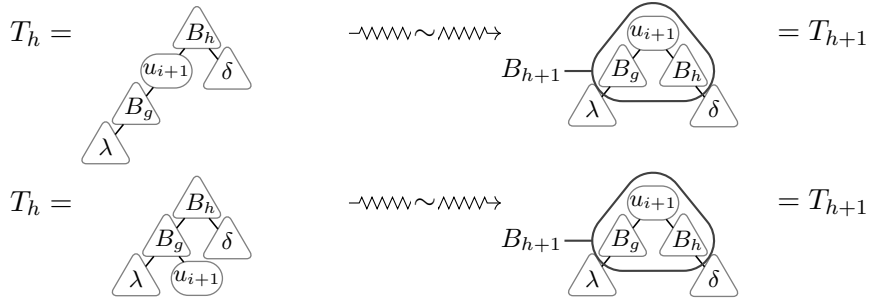
\begin{figure}[t]
    \centering
    \begin{tikzpicture}
        \begin{scope}[tinybst]
        \node[triangle] (aroot) at (0,-32.5mm) {$B_{\mathrlap{h}\,}$}
        child { node (a0) {$u_{i+1}$}
          child { node[triangle] (a00) {$B_{\mathrlap{g}\,}$}
            child { node[triangle] (a000) {$\lambda$} }
            child[missing]
          }
          child[missing]
        }
        child { node[triangle] (a1) {$\delta$} };
      \end{scope}
      \begin{scope}[tinybst]
        \node (broot) at ($ (aroot) + (60mm,0) $) {$u_{i+1}$}
        child { node[triangle] (b0) {$B_{\mathrlap{g}\,}$}
          child { node[triangle] (b00) {$\lambda$} }
          child[missing]
        }
        child { node[triangle] (b1) {$B_{\mathrlap{h}\,}$}
          child[missing]
          child { node[triangle] (b11) {$\delta$} }
        };
      \end{scope}
      \node[font=\small,inner sep=.25mm] (bbi0) at ($ (b0) + (-12mm,0mm) $) {$B_{h+1}$};
      \draw[bstoutline] \convexpath{b1,b0,broot}{4mm};
      \draw[bstoutline] (bbi0) -- ($ (b0) + (-4mm,0mm) $);
      \node[anchor=east] at ($ (aroot) + (-15mm,0) $) {$T_h=$};
      \node[anchor=west] at ($ (broot) + (15mm,0) $) {$=T_{h+1}$};
      \draw ($ (aroot) + (20mm,0) $) edge[mogrifyarrow] node[fill=white,anchor=mid,inner xsep=.25mm,inner ysep=1mm] {$\cyc$} ($ (broot) + (-20mm,0) $);
      \begin{scope}[tinybst]
        \node[triangle] (aroot) at (0,-55mm) {$B_{\mathrlap{h}\,}$}
        child { node[triangle] (a0) {$B_{\mathrlap{g}\,}$}
          child { node[triangle] (a00) {$\lambda$} }
          child { node (a01) {$u_{i+1}$} }
        }
        child { node[triangle] (a1) {$\delta$} };
      \end{scope}
      \begin{scope}[tinybst]
        \node (broot) at ($ (aroot) + (60mm,0) $) {$u_{i+1}$}
        child { node[triangle] (b0) {$B_{\mathrlap{g}\,}$}
          child { node[triangle] (b00) {$\lambda$} }
          child[missing]
        }
        child { node[triangle] (b1) {$B_{\mathrlap{h}\,}$}
          child[missing]
          child { node[triangle] (b11) {$\delta$} }
        };
      \end{scope}
      \node[font=\small,inner sep=.25mm] (bbi0) at ($ (b0) + (-12mm,0mm) $) {$B_{h+1}$};
      \draw[bstoutline] \convexpath{b1,b0,broot}{4mm};
      \draw[bstoutline] (bbi0) -- ($ (b0) + (-4mm,0mm) $);
      \node[anchor=east] at ($ (aroot) + (-15mm,0) $) {$T_h=$};
      \node[anchor=west] at ($ (broot) + (15mm,0) $) {$= T_{h+1}$};
      \draw ($ (aroot) + (20mm,0) $) edge[mogrifyarrow] node[fill=white,anchor=mid,inner xsep=.25mm,inner ysep=1mm] {$\cyc$} ($ (broot) + (-20mm,0) $);
    \end{tikzpicture}
    \caption{Induction step, case~2, two sub-cases.}
    \label{fig:case2}
  \end{figure}%
  By P1, $B_h$ appears at the root of $T_h$; by P2, $B_g$ is next subtree $B_{i_j}$ on the path of left child nodes from
  the root of $T_h$ (and is thus in the left-minimal subtree of $B_h$). By \fullref{Proposition}{prop:infixreading} applied to $U$,
  the symbol $u_{h+1}$ is the unique symbol that is greater than every node of $B_g$ and less than every node of $B_h$. Then the
  node $u_{h+1}$ may be in one of two places in $T_h$, leading to the two sub-cases below. In both cases, as shown in
  \fullref{Figure}{fig:case2}, let $\lambda$ be a reading of the left subtree of $B_g$ and let $\delta$ be a reading of
  the right-maximal subtree of $B_h$; note that the subtrees $B_i$ for $u_i \in U_{h}^\top \setminus \set{u_g,u_h}$
  are contained in $\lambda$.
  \begin{enumerate}
  \item Suppose $u_{h+1}$ is the unique node on the path of left child nodes between $B_g$ and $B_h$. In this case, as
    shown in \fullref[(top)]{Figure}{fig:case2}, $T_h = \psylv{\lambda B_g u_{h+1}\delta B_h}$. Let
    $T_{h+1} = \psylv{\delta B_h\lambda B_g u_{h+1}}$; note that $T_h \cyc T_{h+1}$.

  \item Suppose $u_{h+1}$ is the unique node in the right-maximal subtree of $B_g$ and there are no nodes between $B_g$
    and $B_h$ on the path of left child nodes. In this case, as shown in \fullref[(bottom)]{Figure}{fig:case2},
    $T_h = \psylv{u_{h+1}\lambda B_g \delta B_h}$. Let $T_{h+1} = \psylv{\lambda B_g \delta B_hu_{h+1}}$; note that
    $T_h \cyc T_{h+1}$.
  \end{enumerate}
  In computing $T_{h+1}$, for both sub-cases, the symbol $u_{h+1}$ is inserted first and becomes the root node. Every
  symbol in $B_g$ and $\lambda$ is less than $u_{h+1}$, so these trees are re-inserted into the left subtree of
  $u_{h+1}$. Every symbol in $B_h$ and $\delta$ is greater than $u_{h+1}$, so these trees are re-inserted into the right
  subtree of $u_{h+1}$. Since $B_{h+1}$ consists of $u_{h+1}$ with $B_g$ as its left subtree and $B_h$ as its right
  subtree, the subtree $B_{h+1}$ appears at the root and so $T_{h+1}$ satisfies P1. All the other subtrees $B_i$ for
  $u_i \in U_{h+1}^\top$ are contained in $\lambda$, so $T_{h+1}$ satisfies P2.

  \smallskip
  \noindent\textit{Case 3.} Suppose $u_h$ is the left child of $u_{h+1}$. Then, by the definition of the postfix traversal, $u_{h+1}$
  has empty right subtree in $U$, and so $B_{h+1}$ consists of $u_{h+1}$ with left subtree $B_h$ and right subtree
  empty. (For an illustration of this case, see the step from $T_1$ to $T_2$ in \fullref{Example}{eg:sylvupperbound}.)
  \begin{figure}[t]
    \centering
    \begin{tikzpicture}
      \begin{scope}[tinybst]
        \node[triangle] (aroot) at (0,-30mm) {$B_{\mathrlap{h}\,}$}
        child { node[triangle] (a0) {$\lambda$} }
        child { node[triangle] (a1) {$\delta$}
          child { node (a10) {$u_{i+1}$}
            child[missing]
            child { node[triangle] (a101) {$\beta$} }
          }
          child[missing]
        };
      \end{scope}
      \begin{scope}[tinybst]
        \node (broot) at ($ (aroot) + (60mm,0) $) {$u_{i+1}$}
        child { node[triangle,level distance=7mm] (b0) {$B_{\mathrlap{h}\,}$}
          child { node[triangle] (b00) {$\lambda$} }
          child { node[triangle] (b01) {} }
        }
        child { node[triangle] (b1) {} };
      \end{scope}
      \node[anchor=east,font=\small,inner sep=.25mm] (bbiplus0) at ($ (b0) + (-8mm,0mm) $) {$B_{h+1}$};
      %
      \draw[bstoutline] \convexpath{broot,b0}{4mm};
      \draw[bstoutline] (bbiplus0) -- ($ (b0) + (-4mm,0mm) $);
      \node[anchor=east] at ($ (aroot) + (-15mm,0) $) {$T_h=$};
      \node[anchor=west] at ($ (broot) + (15mm,0) $) {$= T_{h+1}$};
      \draw ($ (aroot) + (20mm,0) $) edge[mogrifyarrow] node[fill=white,anchor=mid,inner xsep=.25mm,inner ysep=1mm] {$\cyc$} ($ (broot) + (-20mm,0) $);
    \end{tikzpicture}
    \caption{Induction step, case 3.}
    \label{fig:case3}
  \end{figure}
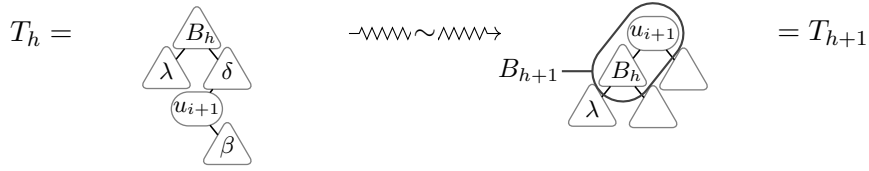%
  Proceeding in a similar way to the previous cases, one sees that, as in \fullref{Figure}{fig:case3},
  $T_h = \psylv{\beta u_{h+1}\delta\lambda B_h}$. Let $T_{h+1} = \psylv{\delta\lambda B_h\beta u_{h+1}}$; then
  $T_h \cyc T_{h+1}$ and $T_{h+1}$ satisfies P1 and P2.

  \smallskip
  \noindent\textit{Case 4.} Suppose that, in $U$, the node $u_h$ is the right child of $u_{h+1}$, and $u_{h+1}$ has
  empty left subtree. (For an illustration of this case, see the step from $T_4$ to $T_5$ in
  \fullref{Example}{eg:sylvupperbound}.) Thus $B_{h+1}$ consists of the node $u_{h+1}$ with empty left subtree and right
  subtree $U_h$.
  \begin{figure}[t]
    \centering
    \begin{tikzpicture}
      \begin{scope}[tinybst]
        \node[triangle] (aroot) at (0mm,-32.5mm) {$B_{\mathrlap{h}\,}$}
        child { node (a0) {$u_{i+1}$}
          child { node[triangle] (a00) {$\lambda$} }
          child[missing]
        }
        child { node[triangle] (a1) {$\delta$} };
      \end{scope}
      \begin{scope}[tinybst]
        \node (broot) at ($ (aroot) + (60mm,0) $) {$u_{i+1}$}
        child { node[triangle] (b0) {$\lambda$} }
        child { node[triangle] (b1) {$B_{\mathrlap{h}\,}$}
          child[missing]
          child { node[triangle] (b11) {$\delta$} }
        };
      \end{scope}
      \node[font=\small,inner sep=.25mm] (bbi0) at ($ (broot) + (-12mm,0mm) $) {$B_{h+1}$};
      \draw[bstoutline] \convexpath{b1,broot}{4mm};
      \draw[bstoutline] (bbi0) -- ($ (broot) + (-4mm,0mm) $);
      \node[anchor=east] at ($ (aroot) + (-15mm,0) $) {$T_h=$};
      \node[anchor=west] at ($ (broot) + (15mm,0) $) {$= T_{h+1}$};
      \draw ($ (aroot) + (20mm,0) $) edge[mogrifyarrow] node[fill=white,anchor=mid,inner xsep=.25mm,inner ysep=1mm] {$\cyc$} ($ (broot) + (-20mm,0) $);
      \begin{scope}[tinybst]
        \node[triangle] (aroot) at (0mm,-50mm) {$B_{\mathrlap{h}\,}$}
        child { node[triangle] (a0) {$\lambda$}
          child[missing]
          child { node (a01) {$u_{i+1}$}
            child { node[triangle] (a010) {$\zeta$} }
            child[missing]
          }
        }
        child { node[triangle] (a1) {$\delta$} };
      \end{scope}
      \begin{scope}[tinybst]
        \node (broot) at ($ (aroot) + (60mm,0) $) {$u_{i+1}$}
        child { node[triangle] (b0) {$\zeta$}
          child { node[triangle] (b00) {$\lambda$} }
          child[missing]
        }
        child { node[triangle] (b1) {$B_{\mathrlap{h}\,}$}
          child[missing]
          child { node[triangle] (b11) {$\delta$} }
        };
      \end{scope}
      \node[font=\small,inner sep=.25mm] (bbi0) at ($ (broot) + (-12mm,0mm) $) {$B_{h+1}$};
      \draw[bstoutline] \convexpath{b1,broot}{4mm};
      \draw[bstoutline] (bbi0) -- ($ (broot) + (-4mm,0mm) $);
      \node[anchor=east] at ($ (aroot) + (-15mm,0) $) {$T_h=$};
      \node[anchor=west] at ($ (broot) + (15mm,0) $) {$= T_{h+1}$};
      \draw ($ (aroot) + (20mm,0) $) edge[mogrifyarrow] node[fill=white,anchor=mid,inner xsep=.25mm,inner ysep=1mm] {$\cyc$} ($ (broot) + (-20mm,0) $);
    \end{tikzpicture}
    \caption{Induction step, case~4, two sub-cases.}
    \label{fig:case4}
  \end{figure}
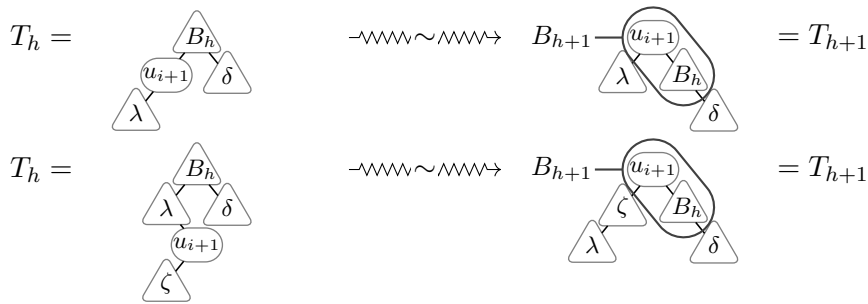%
  Proceeding in a similar way to the previous cases, one sees that there are two sub-cases, as in \fullref{Figure}{fig:case4}:
  \begin{enumerate}
  \item $T_h = \psylv{\lambda u_{h+1}\delta B_h}$. Let $T_{h+1} = \psylv{\delta B_h\lambda u_{h+1}}$.
  \item $T_h = \psylv{\zeta u_{h+1}\lambda \delta B_h}$. Let $T_{h+1} = \psylv{\lambda \delta B_h \zeta u_{h+1}}$.
  \end{enumerate}
  In both sub-cases, $T_h \cyc T_{h+1}$ and $T_{h+1}$ satisfies P1 and P2.


  \medskip
  \noindent\textit{Conclusion.}
  Thus there is a sequence $T = T_0, T_1, \ldots, T_n = U$ with $T_h \cyc T_{h+1}$ and $T_{h+1}$ satisfying P1 and P2
  for $h = 0,\ldots,h-1$. In particular, $T_n$ satisfies P1 and so the subtree $B_n = U$ appears in $T_n$, with its root
  at the root of $T_n$. Hence $T_n = U$.
\end{proof}

\bibliography{\jobname}
\bibliographystyle{alphaabbrv}

\end{document}